\documentclass[11pt]{amsart}

\usepackage{amssymb}
\usepackage{latexcad}

\newtheorem{theorem}{Theorem}[section]
\theoremstyle{plain}

\newtheorem{lemma}[theorem]{Lemma}

\newtheorem{proposition}[theorem]{Proposition}
\newtheorem{remark}[theorem]{Remark}

\numberwithin{equation}{section}

\def\rad{\operatorname{Rad}}
\def\mul{\operatorname{Mul}}
\def\inn{\operatorname{Inn}}
\def\mlt{\operatorname{Mlt}}

\def\im{\operatorname{Im}}

\def\core#1#2{\mathrm{core}_{#1}{#2}}
\def\cl#1{\mathrm{cl}(#1)}

\def\f{\varphi} 

\title[Loops with commuting inner mappings]{Small loops of nilpotency class three with commutative inner mapping groups}

\author{Ale\v{s} Dr\'apal}
\email[Dr\'apal]{drapal@karlin.mff.cuni.cz}
\address[Dr\'apal]{Department of Algebra, Charles University, Sokolovsk\'a 83, 186 75 Prague,
Czech Republic}

\author{Petr Vojt\v{e}chovsk\'y}
\email[Vojt\v{e}chovsk\'y]{petr@math.du.edu}
\address[Vojt\v{e}chovsk\'y]{Department of Mathematics, University of Denver, 2360 S Gaylord St,
Denver, Colorado 80208, USA}

\thanks{A.~Dr\'apal supported by the Institutional Grant MSM 0021620839 and
by the Grant Agency of the Czech Republic 201/09/0296. P.~Vojt\v{e}chovsk\'y supported
by the Enhanced Sabbatical grant of the University of Denver.}

\keywords{Loop, central nilpotency, nilpotent class three, commutative inner mapping group,
symmetric trilinear form, loop of Cs\"org\H{o} type}

\subjclass[2000]{20N05}

\begin{document}

\begin{abstract}
Groups with commuting inner mappings are of nilpotency class at most two, but
there exist loops with commuting inner mappings and of nilpotency class higher
than two, called \emph{loops of Cs\"org\H{o} type}. In order to obtain small
loops of Cs\"org\H{o} type, we expand our programme from \emph{Explicit
constructions of loops with commuting inner mappings},
European~J.~Combin.~\textbf{29} (2008), 1662--1681, and analyze the following
setup in groups:

Let $G$ be a group, $Z\le Z(G)$, and suppose that $\delta:G/Z\times G/Z\to Z$
satisfies $\delta(x,x)=1$, $\delta(x,y)=\delta(y,x)^{-1}$,
$z^{yx}\delta([z,y],x) = z^{xy}\delta([z,x],y)$ for every $x$, $y$, $z\in G$,
and $\delta(xy,z) = \delta(x,z)\delta(y,z)$ whenever $\{x,y,z\}\cap G'$ is not
empty.

Then there is $\mu:G/Z\times G/Z\to Z$ with $\delta(x,y) =
\mu(x,y)\mu(y,x)^{-1}$ such that the multiplication $x*y=xy\mu(x,y)$ defines a
loop with commuting inner mappings, and this loop is of Cs\"org\H{o} type (of
nilpotency class three) if and only if $g(x,y,z) =
\delta([x,y],z)\delta([y,z],x)\delta([z,x],y)$ is nontrivial.

Moreover, $G$ has nilpotency class at most three, and if $g$ is nontrivial then
$|G|\ge 128$, $|G|$ is even, and $g$ induces a trilinear alternating form. We
describe all nontrivial setups $(G,Z,\delta)$ with $|G|=128$. This allows us to
construct for the first time a loop of Cs\"org\H{o} type with an inner mapping
group that is not elementary abelian.
\end{abstract}

\maketitle

\section{Introduction}

Let $Q$ be a loop with neutral element $1$. For $x\in Q$, let $L_x:Q\to Q$,
$y\mapsto xy$ be the \emph{left translation} by $x$, and $R_x:Q\to Q$,
$y\mapsto yx$ the \emph{right translation} by $x$ in $Q$. Then
\begin{displaymath}
    \mlt Q=\langle L_x,\,R_x;\;x\in Q\rangle
\end{displaymath}
is the \emph{multiplication group} of $Q$,
\begin{displaymath}
    \inn Q=\{\varphi\in\mlt{Q};\;\varphi(1)=1\}
\end{displaymath}
is the \emph{inner mapping group} of $Q$, and
\begin{displaymath}
    Z(Q) = \{x\in Q;\;\varphi(x)=x\text{ for all $\varphi\in\inn Q$}\}
\end{displaymath}
is the \emph{center} of $Q$.

Set $Z_1(Q) = Z(Q)$ and define $Z_{i+1}(Q)$ by $Z(Q/Z_i(Q)) = Q/Z_{i+1}(Q)$.
Then $Q$ is \emph{(centrally) nilpotent} if $Z_m(Q)=1$ for some $m$, and the
\emph{nilpotency class} $\cl{Q}$ of $Q$ is the least integer $m$ for which
$Z_m(Q)=1$ occurs.

The \emph{associator subloop} $A(Q)$ of $Q$ is the smallest normal subloop of
$Q$ such that $Q/A(Q)$ is a group. The \emph{derived subloop} $Q'$ is the
smallest normal subloop of $Q$ such that $Q/Q'$ is a commutative group. Set
$Q^{(1)} = Q'$ and $Q^{(i+1)} = (Q^{(i)})'$. Then $Q$ is \emph{solvable} if
$Q^{(m)}=1$ for some $m$.

Questions concerning relations between nilpotency and solvability of $Q$, $\inn
Q$ and $\mlt Q$ go back at least to 1940s and the foundational paper of Bruck,
cf. \cite[p. 278]{BruckTAMS}. Figure \ref{Fg:QIQMQ} summarizes what is
currently known about this problem for \emph{finite} loops. In more detail and
in chronological order:
\begin{enumerate}
\item[$\bullet$] if $Q$ is nilpotent then $\mlt{Q}$ is solvable, by
    \cite[Corollary II, p. 281]{BruckTAMS},
\item[$\bullet$] if $\mlt{Q}$ is nilpotent then $Q$ is nilpotent, by
    \cite[Corollary III, p. 282]{BruckTAMS},
\item[$\bullet$] if $\mlt{Q}$ is solvable then $Q$ is solvable, by
    \cite{Vesanen},
\item[$\bullet$] if $\inn{Q}$ is nilpotent then $\mlt{Q}$ is solvable, by
    \cite{Mazur},
\item[$\bullet$] if $\inn{Q}$ is nilpotent then $Q$ is nilpotent, by
    \cite{Niemenmaa}.
\end{enumerate}
This accounts for all nontrivial implications in Figure \ref{Fg:QIQMQ}.
Moreover, no implications in the figure are missing, as is indicated by the
dotted lines that represent counterexamples and give some information about
their nature. For instance, there exists a Steiner loop $Q$ of order $16$ which
is solvable but $\inn{Q}$ is not solvable.

\setlength{\unitlength}{1.5mm}
\begin{figure}\centering\begin{small}\input{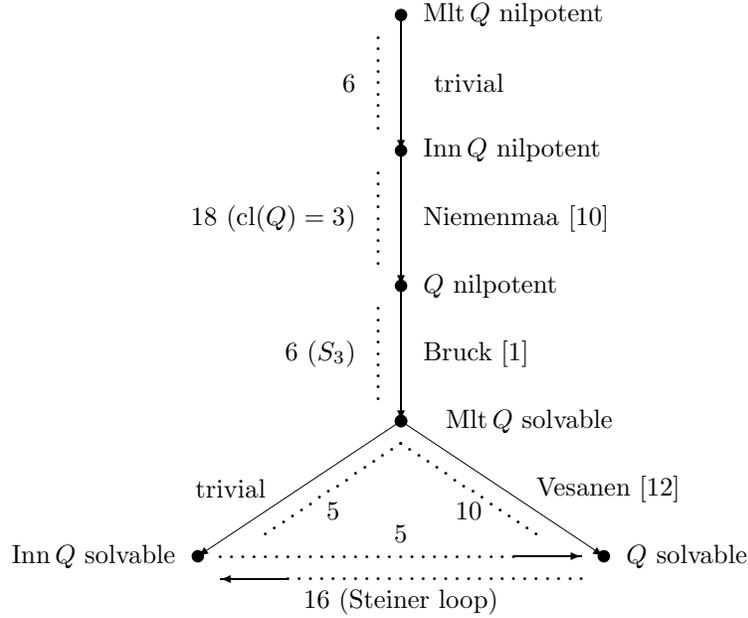}\end{small}
\caption{Implications among nilpotency and solvability of a finite loop $Q$,
its inner mapping group $\inn Q$, and its multiplication group $\mlt Q$.}\label{Fg:QIQMQ}
\end{figure}

Our understanding of the situation is very incomplete, however. If $\cl{Q}=2$
then $\inn{Q}$ is abelian by a result of Bruck \cite{BruckTAMS} but,
conversely, if $\inn{Q}$ is abelian we are still at a loss regarding the
structure of the loop $Q$. For instance, we do not know if there exists a loop
$Q$ with $\inn{Q}$ abelian and $\cl{Q}>3$. We also do not know if there exists
an odd order loop $Q$ with $\inn Q$ abelian and $\cl{Q} = 3$.

We call loops with $\inn{Q}$ abelian and $\cl{Q} \ge 3$ of \emph{Cs\"org\H{o}
type} since Piroska Cs\"org\H{o} gave the first example of such a loop
\cite{Csorgo}. Her example is of order $128$ and no known loop of Cs\"org\H{o}
type has smaller order. In \cite{DV1} we have reconstructed Cs\"org\H{o}'s
example by a method that defines a loop by modifying a group operation. We have
shown that a vast number of such examples can be obtained by this method, even
for the order $128$. All known examples of loops of Cs\"org\H{o} type of order
$128$ are very close to groups in the sense that their associator subloop
consists of only two elements

In this paper we explore the power of the construction invented in \cite{DV1}.
We show that it cannot be used for odd order loops and that it yields no loop
of order less than $128$. We also describe all $125$ groups of order $128$ that
can be used in the construction as the starting point. In all cases we obtain a
loop $Q$ such that $K = Q/A(Q)$ has $K' = Z(K)$ of order $8$. There are $10$
such groups $K$ and all of them can occur in our construction. However, we do
not know if all loops of Cs\"org\H{o} type can be obtained by our construction.

\medskip

The original setup of \cite{DV1} with its complicated subgroup structure (see
\eqref{Eq:OldG}) is recalled in \S \ref{Sc:OriginalSetup}. In the same section
we introduce the much simpler new setup that is based only on the group $G$, a
central subgroup $Z$, and a mapping $\delta:G/Z\times G/Z\to Z$. We show that
every original setup gives rise to a new setup, cf. Proposition
\ref{Pr:OriginalToNew}.

After defining the radical $\rad\f$ and the multiplicative part $\mul\f$ for
general mappings $\f$ in \S \ref{Sc:RadMul}, we investigate the radical and the
multiplicative part of $\delta$ and of the three multiplicative mappings
$f(x,y,z) = \delta([x,y],z)$, $g(x,y,z) = f(x,y,z)f(y,z,x)f(z,x,y)$ and
$h(x,y,z) = [x,[y,z]]$.

In \S \ref{Sc:ConstructLoop} we show how to obtain an original setup from a new
setup. As far as the subgroup structure is concerned, it suffices to take
$R=\rad\delta$ and $N=G'R$, cf. Proposition \ref{Pr:NewToOriginal}. There is
then an obvious choice for the mapping $\mu$ that will turn the new setup into
an original setup. With any valid choice of $\mu$ we have $\cl{G}\le 3$,
$\cl{G[\mu]}\le 3$. Moreover, $\cl{G[\mu]}=3$ if and only if the mapping $g$ is
nontrivial.

\S \ref{Sc:MinProp} is concerned with minimal setups, that is, with nontrivial
new setups with $G$ as small as possible. We know from \cite{DV1} that $|G|\le
128$ in a minimal setup. As a consequence of the Hall-Witt identity, we show
that $|G|$ must be even in a nontrivial setup, and that $\im f$ is a cyclic
group of even order in a minimal setup. A more detailed analysis implies that
$|G|\ge 128$, and that the equality $|G|=128$ holds if and only if we find
ourselves in one of the three scenarios of Theorem \ref{Th:Scenarios}.
Consequently, $\mul\delta = N = G'\rad\delta$ in all minimal setups, once again
showing the relevance of $\mul\delta$ and $\rad\delta$ to the problem at hand.

In \S \ref{Sc:Min} we also show that the three scenarios of Theorem
\ref{Th:Scenarios} do occur, and we describe how to construct all minimal
setups. Proposition \ref{Pr:GModZ} characterizes the groups $G/Z$ that occur in
minimal setups. The groups $G$ that occur in minimal setups can be found in
Subsection \ref{Ss:ComputationalResults}. For each such group $G$ we show how
to construct all minimal setups based on it, starting with the determinant $g$
that uniquely specifies the mapping $f$, which gives rise to $\delta$ via some
free parameters.

There are many examples of loops of Cs\"org\H{o} type of order $128$ due to the
free parameters in the transition from $f$ to $\delta$, and also due to the
numerous choices of $\mu$ for a given $\delta$. Some explicit examples can be
found in \S \ref{Sc:Examples}.

As a byproduct, we construct for the first time a loop $Q$ of Cs\"org\H{o} type
for which $\inn{Q}$ is not elementary abelian. Recall that the structure of
abelian $\inn{Q}$ is a frequently studied but nor very well understood problem,
consisting mostly of nonexistence results. It can be deduced from \cite[Remarks
5.3 and 5.5]{CsorgoKepka} that if $Q$ is a $p$-loop (that is, $|Q|$ is a power
of $p$, $p$ a prime) and $\inn Q$ is abelian then $\inn{Q}$ is a $p$-group. By
\cite{NiemenmaaKepkaJAlg}, $\inn{Q}$ is cyclic if and only if $Q$ is an abelian
group. If $Q$ is a $p$-loop with $\cl{Q}>2$ then $\inn{Q}$ is isomorphic
neither to $\mathbb Z_p\times\mathbb Z_p$, by \cite[Theorem
4.2]{CsorgoJancarikKepka}, nor to $\mathbb Z_p\times\mathbb Z_p\times\mathbb
Z_p$, by \cite{Csorgo}. If $Q$ is finite, $k\ge 2$, $p$ is an odd prime
\cite{NiemenmaaBAMS} or an even prime \cite[Theorem 4.1]{CsorgoJancarikKepka},
then $\inn{Q}$ is not isomorphic to $\mathbb Z_{p^k}\times\mathbb Z_p$. On the
positive side, taking the examples of \cite{DV1} and of this paper into
account, we now know that $\inn{Q}$ can be isomorphic to $(\mathbb Z_2)^6$ or
to $(\mathbb Z_4)^2\times(\mathbb Z_2)^2$ when $Q$ is a $2$-loop with
$\cl{Q}>2$.

\section{The original setup and the new setup}\label{Sc:OriginalSetup}

As in \cite{DV1}, let $G$ be a group and
\begin{equation}\label{Eq:OldG}
    Z\le R\le N\unlhd G,\,Z\le Z(G),\,R\unlhd G,\, G'\le N,\,N/R\le Z(G/R).\tag{$A_0$}
\end{equation}
Let $\mu:G/R\times G/R\to Z$ be a mapping satisfying $\mu(1,x)=\mu(x,1)=1$ for
every $x\in G$, where we write $\mu(x,y)$ instead of the formally correct
$\mu(xR,yR)$. Define $\delta:G/R\times G/R\to Z$ by
\begin{equation}\label{Eq:Mu0}
    \delta(x,y)=\mu(x,y)\mu(y,x)^{-1}.\tag{$A_1$}
\end{equation}
Furthermore, consider the conditions
\begin{align}
    &\mu(xy,z)=\mu(x,z)\mu(y,z)\text{ whenever $\{x,y,z\}\cap N\ne\emptyset$,}\label{Eq:Mu1}\tag{$A_2$}\\
    &\mu(x,yz)=\mu(x,y)\mu(x,z)\text{ whenever $\{x,y,z\}\cap N\ne\emptyset$,}\label{Eq:Mu2}\tag{$A_3$}\\
    &z^{yx}\delta([z,y],x) = z^{xy}\delta([z,x],y)\text{ for every $x$, $y$, $z\in G$.}\label{Eq:D}\tag{$A_4$}
\end{align}

The \emph{original setup} consists of $G$, $Z$, $R$, $N$, $\mu$ and $\delta$ as
above, satisfying \eqref{Eq:OldG}--\eqref{Eq:D}.

\begin{remark}
Note that we have also assumed $N'=1$ as part of \eqref{Eq:OldG} in \cite{DV1}.
This is because we reconstructed the first example of Cs\"{o}rg\H{o} by means
of nuclear extensions, cf. \cite[Section 2]{DV1}, which we understood only with
the added assumption that the normal nuclear subgroup is abelian. But the
assumption $N'=1$ is never used in \S\S 4--7 of \cite{DV1}. In particular, it
is not used in any of the results of \cite{DV1} that we need here.
\end{remark}

From the original setup we can define a loop $G[\mu]$ on $G$ by
\begin{displaymath}
    x*y = xy\mu(x,y),
\end{displaymath}
and we obtain:

\begin{proposition}[\cite{DV1}]\label{Pr:DV1}
Let $Q=G[\mu]$ be the loop obtained from an original setup. Then $\cl G \le 3$,
$\cl Q\le 3$ and $\inn Q$ is an abelian group. Moreover, $\cl{Q}=3$ if and only
if
\begin{equation}\label{Eq:Nontrivial}
    \delta([x,y],z)\delta([y,z],x)\delta([z,x],y)\ne 1\text{ for some $x$, $y$, $z\in G$.}
\end{equation}
\end{proposition}

The condition \eqref{Eq:Nontrivial} can be satisfied even if $\cl{G}=2$, and we
managed to obtain in \cite{DV1} many examples of loops $Q$ of order $128$ with
$\cl{Q}=3$ and $\cl{\inn Q}=1$ from groups of nilpotency class two.

Unlike the original setup, the new setup will be based only on the groups $G$,
$Z\le Z(G)$, and the mapping $\delta$.

Let $G$ be a group and
\begin{equation}\label{Eq:NewG}
    Z\le Z(G).\tag{$B_0$}
\end{equation}
Assume that $\delta:G/Z\times G/Z\to Z$ satisfies
\begin{align}
    &\delta(x,x) = 1\text{ for every $x\in G$,}\label{Eq:D0}\tag{$B_1$}\\
    &\delta(x,y) = \delta(y,x)^{-1}\text{ for every $x$, $y\in G$,}\label{Eq:D1}\tag{$B_2$}\\
    &\delta(xy,z) = \delta(x,z)\delta(y,z)\text{ whenever $\{x$, $y$, $z\}\cap G'\ne\emptyset$,}\label{Eq:D2}\tag{$B_3$}\\
    &z^{yx}\delta([z,y],x) = z^{xy}\delta([z,x],y)\text{ for every $x$, $y$, $z\in G$.}\label{Eq:D3}\tag{$B_4$}
\end{align}
The \emph{new setup} consists of $G$, $Z$ and $\delta$ as above, satisfying the
conditions \eqref{Eq:NewG}--\eqref{Eq:D3}.

\begin{proposition}\label{Pr:OriginalToNew}
Every original setup gives rise to a new setup.
\end{proposition}
\begin{proof}
Let $G$, $Z$, $R$, $N$, $\mu$ and $\delta$ be as in the original setup. We
certainly have \eqref{Eq:NewG}. The mapping $\delta$ is defined modulo $R$ and
hence modulo $Z\le R$. Conditions \eqref{Eq:D0} and \eqref{Eq:D1} follow from
\eqref{Eq:Mu0}. If $\{x,y,z\}\cap G'\ne\emptyset$, we have $\{x,y,z\}\cap
N\ne\emptyset$, and so $\delta(xy,z) =
\mu(xy,z)\mu(z,xy)^{-1}=\mu(x,z)\mu(y,z)\mu(z,x)^{-1}\mu(z,y)^{-1} =
\delta(x,z)\delta(y,z)$, by \eqref{Eq:Mu1} and \eqref{Eq:Mu2}. Finally, the
conditions \eqref{Eq:D} and \eqref{Eq:D3} are identical.
\end{proof}

To show that every new setup gives rise to an original setup, we must find
suitable subgroups $R$ and $N$ and a suitable mapping $\mu$. This is
accomplished in the next three sections, culminating in Proposition
\ref{Pr:New}.

\section{The radical and the multiplicative part}\label{Sc:RadMul}

\subsection{The radical}

Let $(G,\cdot,1)$ be a group and $(H,\cdot,1)$ an abelian group. Let us define
the radical $\rad\f$ for a general mapping $\f:G^2\to H$ satisfying
\begin{equation}\label{Eq:Cocycle}
    \f(1,x)=\f(x,1)=1\text{ for every $x\in G$.}
\end{equation}
Let
\begin{align*}
    &\rad_2 \f = \{t\in G;\;\f(t,x) = \f(x,t) = 1 \text{ for all }x\in G\},\\
    &\rad_1 \f = \{t\in G;\;\f(tx,y)=\f(xt,y)
        =\f(x,ty)=\f(x,yt)=\f(x,y)\text{ for all }x,\,y\in G\}.
\end{align*}
It turns out (cf. Lemma \ref{Lm:Radicals}(ii)) that $\rad_1\f$ is a subgroup of
$G$. We can therefore define
\begin{displaymath}
    \rad \f = \core{G}{(\rad_1 \f)},
\end{displaymath}
where $\core{G}{A}$ is the largest normal subgroup of $G$ contained in $A$,
that is,
\begin{displaymath}
    \core{G}{A} = \bigcap_{g\in G} A^g = \{t\in G;\;t^g\in A\text{ for every }
    g\in G\}.
\end{displaymath}

Call a mapping $\f:G^m\to H$ \emph{multiplicative} if for every $1\le i\le m$
and for every $x_1$, $\dots$, $x_m\in G$ the induced mapping
$\f(x_1,\dots,x_{i-1},-,x_{i+1},\dots,x_m):G\to H$ is a homomorphism. A mapping
$\f:G^m\to H$ is \emph{symmetric} if
\begin{displaymath}
    \f(x_1,\dots,x_m) = \f(x_{\sigma(1)},\dots,x_{\sigma(m)}),
\end{displaymath}
and \emph{alternating} if
\begin{displaymath}
    \f(x_1,\dots,x_m) =
    \f(x_{\sigma(1)},\dots,x_{\sigma(m)})^{\mathrm{sgn}(\sigma)},
\end{displaymath}
for every $x_1$, $\dots$, $x_m\in G$ and every permutation $\sigma$ of
$\{1,\dots,m\}$.

\begin{lemma}\label{Lm:Radicals}
Let $G$ be a group, $H$ an abelian group, and $\f:G^2\to H$ a mapping
satisfying \eqref{Eq:Cocycle}. Then
\begin{enumerate}
\item[(i)] $\rad_1 \f\subseteq \rad_2 \f$,
\item[(ii)] $\rad_1 \f\le G$, $\rad \f = \{t\in G;\; t^z\in \rad_1 \f$ for
    every $z\in G\} \unlhd G$,
\item[(iii)] if $\f$ is multiplicative then $G'\le \rad\f = \rad_1 \f
    =\rad_2 \f$.
\end{enumerate}
\end{lemma}
\begin{proof}
Assume that $t\in\rad_1\f$. Then $\f(t,x)=\f(t\cdot 1,x) = \f(1,x) = 1$ for
every $x\in G$, and, similarly, $\f(x,t)=1$. Hence $t\in\rad_2\f$.

To establish (ii), assume that $t$, $s\in\rad_1\f$. Then $\f(tsx,y) = \f(sx,y)
= \f(x,y)$, and similarly for all other defining conditions of $\rad_1\f$, so
$ts\in\rad_1\f$. We have $\f(t^{-1}x,y) = \f(tt^{-1}x,y) = \f(x,y)$, where the
first equality follows from $t\in\rad_1\f$. Hence $t^{-1}\in\rad_1\f$. The rest
of (ii) is clear.

Assume that $\f$ is multiplicative and let $t\in\rad_2\f$. Then $\f(tx,y) =
\f(t,y)\f(x,y) = \f(x,y)$, and similarly for the other defining conditions of
$\rad_1\f$, so $\rad_1\f=\rad_2\f$. To show that $\rad\f = \rad_1\f$, it
therefore suffices to prove that $\rad_2\f\unlhd G$. But for $t\in \rad_2\f$
and $x$, $z\in G$, we have
\begin{multline*}
    \f(t^z,x) = \f(z^{-1},x)\f(t,x)\f(z,x) = \f(z^{-1},x)\f(z,x)\f(t,x)\\
    =\f(z^{-1}z,x)\f(t,x) = \f(1,x)\f(t,x) = \f(t,x) = 1.
\end{multline*}
A similar argument shows $\f(x,t^z) = 1$, hence $\rad\f = \rad_1\f = \rad_2\f$.
Now, $\f([x,y],z) = \f(x^{-1},z)\f(y^{-1},z)\f(x,z)\f(y,z)=1$, so $G'\le
\rad\f$.
\end{proof}

We remark that it is not difficult to construct $\f:G^2\to H$ such that
$\rad_2\f$ is not a subgroup of $G$, or such that $\rad_1\f$ is not a normal
subgroup of $G$. It is therefore necessary to transition from $\rad_1\f$ to its
core $\rad\f$ if we wish to consider the mapping $\f$ modulo its radical.

More precisely, $\f:G^2\to H$ satisfying \eqref{Eq:Cocycle} induces a mapping
$\f:(G/\rad\f)^2\to H$ with trivial radical, and we will often identity this
induced mapping with $\f$.

Since we will also need the radical for multiplicative mappings $\f:G^3\to H$,
we set
\begin{equation}\label{Eq:Rad3Mult}
    \rad\f=\{t\in G;\;\f(t,x,y)=\f(x,t,y)=\f(x,y,t)=1\text{ for every }x,\,y\in
    G\}
\end{equation}
in such a case. It is then easy to see that $G'\le\rad\f\unlhd G$ and that $\f$
is well-defined modulo $\rad\f$ once again.

\subsection{The multiplicative part}

For a group $G$, an abelian group $H$ and a mapping $\f:G^2\to H$ satisfying
\eqref{Eq:Cocycle} define the \emph{multiplicative part} $\mul\f$ of $\f$ by
\begin{align*}
    \mul \f = \{t\in G;\;&\f(tx,y)=\f(t,y)\f(x,y)=\f(xt,y),\\
        &\f(x,ty) = \f(x,t)\f(x,y)=\f(x,yt),\\
        &\f(xy,t)=\f(x,t)\f(y,t),\\
        &\f(t,xy)=\f(t,x)\f(t,y)\text{ for every $x$, $y\in G$}\}.
\end{align*}

\begin{lemma}\label{Lm:Mul}
Let $G$ be a group, $H$ an abelian group, and $\f:G^2\to H$ a mapping
satisfying \eqref{Eq:Cocycle}. Then
\begin{enumerate}
\item[(i)] $\mul\f\le G$,
\item[(ii)] $\rad\f\le \rad_1\f\le\mul\f$,
\item[(iii)] if $\f$ is multiplicative then $\mul\f=G$.
\end{enumerate}
\end{lemma}
\begin{proof}
To show (i), assume that $t$, $s\in\mul\f$. Then $\f(tsx,y) = \f(t,y)\f(sx,y) =
\f(t,y)\f(s,y)\f(x,y) = \f(ts,y)\f(x,y)$, and the conditions $\f(ts,y)\f(x,y) =
\f(xts,y)$, $\f(x,tsy) = \f(x,ts)\f(x,y)\f(x,yts)$ are established similarly.
We also have $\f(xy,ts) = \f(xy,t)\f(xy,s) = \f(x,t)\f(y,t)\f(x,s)\f(y,s) =
\f(x,ts)\f(y,ts)$, hence $ts\in\mul\f$. Now, $\f(t^{-1},y)\f(tx,y) =
\f(t^{-1},y)\f(t,y)\f(x,y) = \f(t^{-1}t,y)\f(x,y) = \f(1,y)\f(x,y) = \f(x,y) =
\f(t^{-1}tx,y)$ for every $x$, $y\in G$, and hence $\f(t^{-1}x,y) =
\f(t^{-1},y)\f(x,y)$ for every $x$, $y\in G$. The rest of (i) is similar.

To prove (ii), let $t\in\rad_1\f$. Then $t\in\rad_2\f$ by Lemma
\ref{Lm:Radicals}, and so $\f(tx,y) = \f(x,y) = \f(t,y)\f(x,y)$ for every $x$,
$y\in G$. The conditions $\f(xt,y)=\f(t,y)\f(x,y)$ and $\f(x,ty) =
\f(x,t)\f(x,y)=\f(x,yt)$ follow by a similar argument. The final two conditions
$\f(xy,t)=\f(x,t)\f(y,t)$, $\f(t,xy)=\f(t,x)\f(t,y)$ hold trivially, as
$t\in\rad_2\f$.

Part (iii) is obvious.
\end{proof}

\begin{lemma}\label{Lm:RadMul}
Let $G$ be a group, $H$ an abelian group, and $\f:G^2\to H$ a mapping
satisfying \eqref{Eq:Cocycle}. Assume that $G'\le\mul\f$. Then
\begin{enumerate}
\item[(i)] $\mul\f\unlhd G$,
\item[(ii)] $\mul\f/\rad\f \le Z(G/\rad\f)$.
\end{enumerate}
\end{lemma}
\begin{proof}
Let $t\in \mul\f$, $z\in G$. For (i), we want to show that $t^z\in\mul\f$. We
have, $\f(t^zx,y) = \f([z,t^{-1}]tx,y) = \f([z,t^{-1}],y)\f(t,y)\f(x,y) =
\f([z,t^{-1}]t,y)\f(x,y) = \f(t^z,y)\f(x,y)$. The remaining defining conditions
of $\mul \f$ are verified analogously.

To prove (ii) it suffices to show that $[m,u]\in\rad\f$ for every $m\in\mul\f$,
$u\in G$. By Lemma \ref{Lm:Radicals}, it suffices to show that
$[m,u]^z\in\rad_1\f$. We prove $\f([m,u]^zx,y) = \f(x,y)$, and leave the rest
to the reader. First, $\f([m,u]^zx,y) = \f([m^z,u^z]x,y) =
\f([m^z,u^z],y)\f(x,y)$. Second, for every $t\in\mul\f$ we have $\f(t,y)\f(x,y)
= \f(tx,y) = \f(xt^x,y) = \f(x,y)\f(t^x,y)$ by (i), and so $\f(t^x,y) =
\f(t,y)$, $\f([t,x],y) = \f(t^{-1}t^x,y) = \f(t^{-1},y)\f(t^x,y) =
\f(t^{-1},y)\f(t,y) = \f(t^{-1}t,y) = 1$. Finally, as $m^z=t\in\mul\f$, we get
$\f([m^z,u^z],y) = 1$.
\end{proof}

\section{Three multiplicative mappings}\label{Sc:Three}

Suppose that $G$, $Z$ and $\delta$ form a new setup. Since $\delta$ is defined
modulo $Z\unlhd G$, we see that $Z\le\rad\delta$ when $\delta$ is considered as
a mapping $G\times G\to Z$. The conditions \eqref{Eq:D1}, \eqref{Eq:D2} imply
that $G'\le\mul\delta$. We will use these properties of $\rad\delta$ and
$\mul\delta$ throughout.

Define $f$, $g$, $h:G^3\to Z$ by
\begin{align}
    f(x,y,z) &= \delta([x,y],z),\notag\\
    g(x,y,z) &= f(x,y,z)f(y,z,x)f(z,x,y),\label{Eq:fgh}\\
    h(x,y,z) &= [x,[y,z]],\notag
\end{align}
where $x$, $y$, $z\in G$. (See Lemma \ref{Lm:fgh}(iv) for $\im h\le Z$.) These
three mappings and their radicals play a crucial role in the analysis of the
new setup. Note that all three mappings are well-defined modulo $\rad\delta$.

\begin{lemma}\label{Lm:fgh}
In the new setup $G$, $Z$, $\delta$ the following conditions are satisfied:
\begin{enumerate}
\item[(i)] $z^{-xy}z^{yx} = [z,[y^{-1},x^{-1}]]$, $[G,G']\le Z$,
    $\cl{G/Z}\le 2$, $\cl{G}\le 3$, and $f(z,x,y) =
    h(z,y,x)f(z,y,x)$,
\item[(ii)] $f$ is multiplicative, $\mul\delta\le\rad f$, and
    $f(x,y,z)=f(y,x,z)^{-1}$,
\item[(iii)] $g$ is multiplicative, $\rad f\le\rad g$, and $g(x,y,z) =
    g(y,z,x)=g(y,x,z)^{-1}$,
\item[(iv)] $h$ is multiplicative, $\im h\le Z$, $\rad f\le\rad h$, and
    $h(x,y,z) = h(x,z,y)^{-1}$,
\item[(v)] $G/\rad g$ is an elementary abelian $2$-group, and the mapping
    $g:G^3\to Z$ induces a trilinear alternating mapping $(G/\rad g)^3\to
    \{a\in Z;\;a^2=1\}$.
\end{enumerate}
\end{lemma}
\begin{proof}
(i) By \eqref{Eq:D3}, $z^{-xy}z^{yx} = f(z,x,y)f(z,y,x)^{-1}\in Z\le Z(G)$.
Since $z^{-xy}z^{yx} = [z,[y^{-1},x^{-1}]]^{xy}$ holds in any group, we have
$[z,[y^{-1},x^{-1}]] = f(z,x,y)f(z,y,x)^{-1}\in Z\le Z(G)$. Thus $[G,G']\le Z$,
$\cl{G/Z}\le 2$, and $\cl{G}\le 3$. In any group of nilpotency class three we
have $[z,[y^{-1},x^{-1}]] = [z,[y,x]]$, so $h(z,y,x) = [z,[y,x]] =
[z,[y^{-1},x^{-1}]] = f(z,x,y)f(z,y,x)^{-1}$.

(ii) $f$ is well-defined modulo $\rad\delta$ and $\cl{G/\rad\delta}\le 2$ by
(i). Since $[xy,z]=[x,z][y,z]$ in any group of nilpotency class $2$, we have
$f(xy,z,w) = \delta([xy,z],w) = \delta([x,z][y,z],w) =
\delta([x,z],w)\delta([y,z],w) = f(x,z,w)f(y,z,w)$. Note that $f(x,yz,w) =
f(x,y,w)f(x,z,w)$ follows similarly. Finally, using $G'\le\mul\delta$, we have
$f(x,y,zw) = \delta([x,y],zw) = \delta([x,y],z)\delta([x,y],w) =
f(x,y,z)f(x,y,w)$.

We have shown that $f$ is multiplicative, so $\rad f$ is defined by \eqref{Eq:Rad3Mult}. Let
$t\in\mul\delta$. Then
\begin{multline*}
    f(x,y,t)=\delta([x,y],t) =
    \delta(x^{-1},t)\delta(y^{-1},t)\delta(x,t)\delta(y,t) =\\
    \delta(x^{-1},t)\delta(x,t)\delta(y^{-1},t)\delta(y,t) =
    \delta(x^{-1}x,t)\delta(y^{-1}y,t)=1.
\end{multline*}
Moreover, by Lemma \ref{Lm:RadMul}(ii), $[t,x]\in\rad\delta$, and we have
$f(t,x,y) = \delta([t,x],y)=1$. The equality $f(x,t,y)=1$ follows similarly.
Thus $\mul\delta\le\rad f$.

Finally, by multiplicativity of $f$ and by $\delta([x,x],y)=1$, we have $1=
f(xy,xy,z) = f(x,y,z)f(y,x,z)$.

(iii) The mapping $g$ is multiplicative since $f$ is, and it is clearly
invariant under a cyclic shift of its arguments. If $t\in\rad f$, we have
$g(t,x,y)=f(t,x,y)f(x,y,t)f(y,t,x)=1$ for every $x$, $y\in G$, so $t\in\rad g$.
By (ii), $g(x,y,z) = f(x,y,z)f(y,z,x)f(z,x,y) =
f(y,x,z)^{-1}f(z,y,x)^{-1}f(x,z,y)^{-1} = (f(y,x,z)f(x,z,y)f(z,y,x))^{-1} =
g(y,x,z)^{-1}$.

(iv) By (i) and (ii), $h$ is multiplicative and $\im h\le Z$. Assume that
$t\in\rad f$. Then $h(t,x,y) = f(t,y,x)f(t,x,y)^{-1}=1$ and, similarly,
$h(x,t,y) = h(x,y,t)=1$, so $t\in\rad h$. Since $1=[y,[x,x]]=h(y,x,x)$, we have
$h(x,y,z)=h(x,z,y)^{-1}$ by multiplicativity.

(v) Let us first show that $g(x,y,z)^2=1$ for every $x$, $y$, $z\in G$. In view
of (iii), this is equivalent to $g(x,y,z) = g(y,x,z)$. Now,
\begin{align*}
    g(y,x,z) &= f(y,x,z)f(x,z,y)f(z,y,x)\\
        &=   f(y,z,x)[y,[z,x]]f(x,y,z)[x,[y,z]]f(z,x,y)[z,[x,y]]\\
        &=   g(x,y,z)[x,[y,z]][y,[z,x]][z,[x,y]],
\end{align*}
by (iv). The Hall-Witt identity
\begin{displaymath}
    [[x,y^{-1}],z]^y[[y,z^{-1}],x]^z[[z,x^{-1}],y]^x=1
\end{displaymath}
valid in all groups yields $[x,[y,z]][y,[z,x]][z,[x,y]]=1$ in groups of
nilpotency class $3$, so we are done by (i).

The established identity $g(x,y,z)^2=1$ shows not only that the image of $g$ is
contained in the vector space $U=\{a\in Z;\;a^2=1\}$, but also that $G/\rad{g}$
is of exponent two, since for every $x\in G$ we have $g(x^2,y,z) =
g(x,y,z)^2=1$, by multiplicativity of $g$.
\end{proof}

\section{Constructing loops of Cs\"org\H{o} type from the new
setup}\label{Sc:ConstructLoop}

\begin{proposition}\label{Pr:NewToOriginal}
Suppose that $G$, $Z$, $\delta$ is a new setup. Set $R=\rad \delta$ and $N =
G'R$. Then $\delta$ is well-defined as a mapping $\delta:G/R\times G/R\to Z$,
and \eqref{Eq:OldG} holds.
\end{proposition}
\begin{proof}
By Lemma \ref{Lm:Radicals}, $Z\le R\unlhd G$ and $\delta$ is well defined
modulo $R$. Obviously, $R\le N$, $G'\le N$, and $N\unlhd G$. Since $[G,G']\le
Z\le R$ by Lemma \ref{Lm:fgh}(i), we have $N/R=G'R/R\le Z(G/R)$.
\end{proof}

Suppose that $G$, $Z$, $\delta$ is a new setup, and let $R=\rad\delta$,
$N=G'R$, $\overline{G}=G/R$, $\overline{N}=N/R$. In view of Proposition
\ref{Pr:NewToOriginal}, to obtain an original setup from $G$ we only need to
construct $\mu:\overline{G}\times \overline{G}\to Z$ so that
\eqref{Eq:Mu0}--\eqref{Eq:Mu2} hold. We show how to construct all such mappings
$\mu$ in principle (by rephrasing the problem in terms of compatible parameter
sets below), and we construct one $\mu$ explicitly (see Lemma
\ref{Lm:DefaultP}).

Let $T=\{t_1=1,\dots,t_n\}$ be a transversal to $\overline{N}$ in
$\overline{G}$. We say that
\begin{equation}\label{Eq:CPS}
    \mathcal P = \{\psi_k;\;k\in \overline{N}\}\cup\{\varphi_i;\;1\le i\le n\}\cup\{\tau_{ij};\;1\le i\le j\le n\}
\end{equation}
is a \emph{compatible parameter set} if
\begin{enumerate}
\item[$\bullet$] $\psi_k:\overline{G}\to Z$ is a homomorphism for every
    $k\in \overline{N}$,
\item[$\bullet$] $\psi_{k'}(k) = \delta(k',k)\psi_k(k')$ for every $k$,
    $k'\in \overline{N}$,
\item[$\bullet$] $\varphi_i:\overline{N}\to Z$ is a homomorphism for every
    $1\le i\le n$,
\item[$\bullet$] $\psi_k(t_i) = \varphi_i(k)$ for every $k\in \overline{N}$
    and $1\le i\le n$,
\item[$\bullet$] $\tau_{ij}\in G$ for every $1\le i\le j\le n$, and
    $\tau_{1i} = 1$ for every $1\le i\le n$.
\end{enumerate}

Note that $\psi_k(1)=1$, $\psi_1(k) = \delta(1,k)\psi_k(1)=1$ and $\varphi_1(k)
= \psi_k(1) = 1$ holds for every $k\in\overline{N}$ in a compatible parameter
set.

\begin{lemma}\label{Lm:DefaultP}
In a new setup $G$, $Z$, $\delta$, let $R=\rad\delta$, $N=G'R$,
$\overline{G}=G/R$, $\overline{N}=N/R$, and let $T=\{t_1=1$, $\dots$, $t_n\}$
be a transversal to $\overline{N}$ in $\overline{G}$. For every
$k\in\overline{N}$, let $\psi_k = \delta(k,-)$. For every $1\le i\le n$, let
$\varphi_i=\delta(-,t_i)$. For every $1\le i\le j\le n$, let $\tau_{ij} =
\delta(t_i,t_j)$. Then $\mathcal P$ of \eqref{Eq:CPS} is a compatible parameter
set.
\end{lemma}
\begin{proof}
We have $G'\le\mul\delta$, $R\le\mul\delta$ by Lemma \ref{Lm:Mul}, and so $N\le
\mul\delta$. Then $\psi_k$ and $\varphi_i$ are homomorphisms. For $k$, $k'\in
N$, the equality $\psi_{k'}(k) = \delta(k',k)\psi_k(k')$ is equivalent to
$\delta(k,k')=1$. Now, $\delta(N,N) = \delta(G'R,G'R) = \delta(G',G')$ and
$\delta([x,y],G') = f(x,y,G')=1$ by Lemma \ref{Lm:fgh}(ii). Finally, the
equality $\psi_k(t_i) = \varphi_i(k)$ holds trivially.
\end{proof}

Given a compatible parameter set $\mathcal P$, define $\mu_{\mathcal
P}=\mu:\overline{G}\times \overline{G}\to Z$ as follows:
\begin{align*}
    &\mu(k,k') = \psi_k(k')\text{ for every $k$, $k'\in \overline{N}$},\\
    &\mu(k,t_i) = \varphi_i(k)\text{ for every $k\in \overline{N}$, $1\le i\le n$},\\
    &\mu(t_i,k) = \delta(t_i,k)\mu(k,t_i)\text{ for every $k\in \overline{N}$, $1\le i\le n$},\\
    &\mu(t_i,t_j) = \tau_{ij}\text{ for every $1\le i\le j\le n$},\\
    &\mu(t_j,t_i) = \delta(t_j,t_i)\mu(t_i,t_j)\text{ for every $1\le i\le j\le n$},\\
    &\mu(kt,k't') = \mu(k,k')\mu(k,t')\mu(t,k')\mu(t,t')\text{ for every $k$, $k'\in \overline{N}$, $t$, $t'\in T$}.
\end{align*}
Observe that $\mu_{\mathcal P}$ is well-defined. (The first five lines do not
lead to a contradiction: $\mu(k,1) = \psi_k(1)=1$ according to the first line
and $\mu(k,1) = \varphi_1(k)=1$ according to the second line, $\mu(1,k) =
\psi_1(k)=1$ according to the first line and $\mu(1,k) = \delta(1,k)\mu(k,1) =
1$ according to the third line. We also have $\mu(1,t_i)=\varphi_i(1)=1$ by the
second line and $\mu(1,t_i) = \tau_{1i}=1$ by the fourth line, $\mu(t_i,1) =
\delta(t_i,1)\mu(1,t_i)=1$ by the third line and $\mu(t_i,1) =
\delta(t_i,1)\mu(1,t_i)=1$ by the fifth line. Finally, $\mu(t_i,t_i) =
\delta(t_i,t_i)\mu(t_i,t_i)$ in the fifth line is sound thanks to
\eqref{Eq:D0}. Adding the sixth line also does not lead to a contradiction:
$\mu(k1,k'1) = \mu(k,k')\mu(k,1)\mu(1,k')\mu(1,1) = \mu(k,k')$,
$\mu(k1,1t_i) = \mu(k,1)\mu(k,t_i)\mu(1,1)\mu(1,t_i) =
\mu(k,t_i)$, similarly for $\mu(t_i,k)$, and we have $\mu(1t_i,1t_j) =
\mu(1,1)\mu(1,t_j)\mu(t_i,1)\mu(t_i,t_j) = \mu(t_i,t_j)$.)

\begin{lemma}\label{Lm:MuP}
Let $G$, $Z$, $\delta$ be a new setup, $R=\rad\delta$, $N=G'R$,
$\overline{G}=G/R$, $\overline{N}=N/R$, and let $\mathcal P$ be a compatible
parameter set. Then $\mu=\mu_{\mathcal P}$ satisfies
\eqref{Eq:Mu0}--\eqref{Eq:Mu2}.
\end{lemma}
\begin{proof}
It suffices to show \eqref{Eq:Mu0} and \eqref{Eq:Mu1}, as \eqref{Eq:Mu2} is a
consequence: if $\{x,y,z\}\cap \overline{N}\ne\emptyset$, we have
\begin{multline*}
    \mu(x,yz) = \delta(x,yz)\mu(yz,x) = \delta(yz,x)^{-1}\mu(yz,x)\\ =
    \delta(y,x)^{-1}\delta(z,x)^{-1}\mu(y,x)\mu(z,x)=
    \delta(x,y)\mu(y,x)\delta(x,z)\mu(z,x) = \mu(x,y)\mu(x,z).
\end{multline*}

Let us prove \eqref{Eq:Mu0}. We have $\mu(k,k') = \psi_k(k') =
\delta(k,k')\psi_{k'}(k) = \delta(k,k')\mu(k',k)$. Since $\mu(t_i,k) =
\delta(t_i,k)\mu(k,t_i)$ for every $k\in \overline{N}$ and $1\le i\le n$, we
have $\mu(k,t_i) = \delta(t_i,k)^{-1}\mu(t_i,k) = \delta(k,t_i)\mu(t_i,k)$,
too. Also, $\mu(t_j,t_i) = \delta(t_j,t_i)\mu(t_i,t_j)$ for every $1\le i\le
j\le n$, which yields $\mu(t_i,t_j) = \delta(t_j,t_i)^{-1}\mu(t_j,t_i) =
\delta(t_i,t_j)\mu(t_i,t_j)$, too. Using the already established equalities, we
have $\mu(kt,k't') = \mu(k,k')\mu(k,t')\mu(t,k')\mu(t,t') =
\delta(k,k')\mu(k',k)\delta(k,t')\mu(t',k)\delta(t,k')\mu(k',t)\delta(t,t')\mu(t',t)
= \delta(kt,k't')\mu(k't',kt)$ for every $k$, $k'\in \overline{N}$, $t$, $t'\in
T$, which is \eqref{Eq:Mu0}.

We split the proof of \eqref{Eq:Mu1} into three cases, depending on which of
the arguments $x$, $y$, $z$ in \eqref{Eq:Mu1} belongs to $\overline{N}$. In all
situations, we will assume that $k$, $k'$, $k''\in \overline{N}$ and $t$, $t'$,
$t''\in T$.

Case $x\in \overline{N}$. We have $\mu(kk't', k''t'') =
\mu(kk',k'')\mu(kk',t'')\mu(t',k'')\mu(t',t'')$ and also
$\mu(k,k''t'')\mu(k't',k''t'') =
\mu(k,k'')\mu(k,t')\mu(k',k'')\mu(k',t'')\mu(t',k'')\mu(t',t'')$, so we need to
check that $\mu(kk',t'') = \mu(k,t'')\mu(k',t'')$, or, with $t''=t_i$,
$\varphi_i(kk') = \varphi_i(k)\varphi_i(k')$, which is true.

Case $y\in \overline{N}$. Recall that $\overline{N}\le Z(\overline{G})$. Thus
$\mu(ktk',k''t'') = \mu(kk't,k''t'')$, while $\mu(kt,k''t'')\mu(k',k''t'') =
\mu(k',k''t'')\mu(kt,k''t'') = \mu(k'kt,k''t'') = \mu(kk't,k''t'')$ by the case
$x\in \overline{N}$.

Case $z\in \overline{N}$. Using $\overline{N}\le Z(\overline{G})$ and the case
$x\in \overline{G}$ once again, we have $\mu(ktk't',k'') =
\mu(kk'tt',k'')=\mu(kk',k'')\mu(tt',k'') = \mu(k,k'')\mu(k',k'')\mu(tt',k'')$,
$\mu(kt,k'')\mu(k't',k'') = \mu(k,k'')\mu(t,k'')\mu(k',k'')\mu(t',k'')$, so we
need to check that $\mu(tt',k'') = \mu(t,k'')\mu(t',k'')$. By \eqref{Eq:Mu0}
this is equivalent to $\delta(tt',k'')\mu(k'',tt') =
\delta(t,k'')\mu(k'',t)\delta(t',k'')\mu(k'',t')$, or
$\mu(k'',tt')=\mu(k'',t)\mu(k'',t')$, or $\psi_{k''}(tt') =
\psi_{k''}(t)\psi_{k''}(t')$, which is true.
\end{proof}

The converse of Lemma \ref{Lm:MuP} is also true:

\begin{lemma} Let $G$, $Z$, $\delta$ be a new setup, $R=\rad\delta$, $N=G'R$,
$\overline{G}=G/R$, $\overline{N}=N/R$. Suppose that $\mu:\overline{G}\times
\overline{G}\to Z$ satisfies \eqref{Eq:Mu0}--\eqref{Eq:Mu2}. Then there exists
a compatible parameter set $\mathcal P$ such that $\mu=\mu_{\mathcal P}$.
\end{lemma}
\begin{proof}
For every $k\in \overline{N}$, let $\psi_k = \mu(k,-):\overline{G}\to Z$. By
\eqref{Eq:Mu2}, $\psi_k$ is a homomorphism. Moreover, $\psi_{k'}(k) = \mu(k',k)
= \delta(k',k)\mu(k,k') = \delta(k',k)\psi_k(k')$ by \eqref{Eq:Mu0}. For every
$1\le i\le n$, let $\varphi_i$ be the restriction of $\mu(-,t_i)$ to
$\overline{N}$. By \eqref{Eq:Mu1}, $\varphi_i$ is a homomorphism. Moreover,
$\psi_k(t_i) = \mu(k,t_i) = \varphi_i(k)$ for every $k\in \overline{N}$, $1\le
i\le n$. Finally, for $1\le i\le j\le n$, let $\tau_{ij} = \mu(t_i,t_j)$. Then
$\mathcal P=\{\psi_k;\;k\in \overline{N}\} \cup \{\varphi_i;\;1\le i\le n\}
\cup \{t_{ij};\;1\le i\le j\le n\}$ is a compatible parameter set, and we can
use it to obtain $\nu = \mu_{\mathcal P}:\overline{G}\times \overline{G}\to Z$.
It remains to show that $\nu=\mu$.

We have $\nu(k,k') = \psi_k(k') = \mu(k,k')$ for every $k$, $k'\in
\overline{N}$; $\nu(k,t_i) = \varphi_i(k) = \mu(k,t_i)$ for every $k\in
\overline{N}$, $1\le i\le n$; $\nu(t_i,k) = \delta(t_i,k)\nu(k,t_i) =
\delta(t_i,k)\mu(k,t_i) = \mu(t_i,k)$ for every $k\in \overline{N}$, $1\le i\le
n$, by \eqref{Eq:Mu0}; $\nu(t_i,t_j) = \tau_{ij} = \mu(t_i,t_j)$ for every
$1\le i\le j\le n$; $\nu(t_j,t_i) = \delta(t_j,t_i)\nu(t_i,t_j) =
\delta(t_j,t_i)\mu(t_i,t_j) = \mu(t_j,t_i)$ for every $1\le i\le j\le n$, by
\eqref{Eq:Mu0}; and, finally, $\nu(kt,k't') =
\nu(k,k')\nu(k,t')\nu(t,k')\nu(t,t') = \mu(k,k')\mu(k,t')\mu(t,k')\mu(t,t') =
\mu(kt,k't')$, by \eqref{Eq:Mu1} and \eqref{Eq:Mu2}.
\end{proof}

Summarizing:

\begin{proposition}\label{Pr:New}
Assume that $G$, $Z$, $\delta$ form a new setup, and let $R=\rad\delta$,
$N=G'R$. Let $\mu:G/R\times G/R\to Z$ be a mapping satisfying
\eqref{Eq:Mu0}--\eqref{Eq:Mu2}, which is guaranteed to exist. Then $G$, $Z$,
$R$, $N$, $\mu$, $\delta$ form an original setup.

In particular, with $Q=G[\mu]$, we have $\cl{G}\le 3$, $\cl{Q}\le 3$, and
$\cl{\inn Q}=1$. Moreover, $\cl{Q}=3$ if and only if \eqref{Eq:Nontrivial} is
satisfied.
\end{proposition}
\begin{proof}
By Proposition \ref{Pr:OriginalToNew}, \eqref{Eq:OldG} holds. With the
compatible parameter set $\mathcal P$ of Lemma \ref{Lm:DefaultP}, the mapping
$\mu=\mu_{\mathcal P}$ satisfies \eqref{Eq:Mu0}--\eqref{Eq:Mu2}, by Lemma
\ref{Lm:MuP}. Hence $G$, $Z$, $R$, $N$, $\mu$, $\delta$ form an original setup.
Then $\cl{G}$, $\cl{Q}$ and $\cl{\inn{Q}}$ depend only on $\delta$, and are as
in Proposition \ref{Pr:DV1}.
\end{proof}

\section{Structural properties of minimal setups}\label{Sc:MinProp}

A new setup $G$, $Z$, $\delta$ with $g$ as in \eqref{Eq:fgh} is said to be
\emph{nontrivial} if $g\ne 1$. By Proposition \ref{Pr:New}, the associated loop
$Q=G[\mu]$ satisfies $\cl{Q}=3$ if and only if the setup is nontrivial. A
nontrivial setup is said to be \emph{minimal} if $|G|$ is as small as possible.

In \cite{DV1}, the original setup yielded many loops $Q$ of order $128$ with
$\inn Q$ abelian and $\cl{Q}=3$. Hence we can assume $|G|\le 128$ in a minimal
setup.

In this section we show that $|G|$ is even in any nontrivial setup, that
$|G|=128$ in a minimal setup, and obtain many structural results that allow us
to describe all minimal setups in the next section.

The investigation will be guided by the already established inclusions
\begin{equation}\label{Eq:Inclusions}
    G\ge\rad g\ge\rad f\ge\mul\delta\ge\rad\delta\ge Z
\end{equation}
valid in any new setup, cf. Lemmas \ref{Lm:Mul}, \ref{Lm:fgh}. It is natural to
work with $\mul\delta$ rather than $N=G'\rad\delta\le \mul\delta$ here.

\begin{lemma}\label{Lm:ZCyclic}
In a minimal setup $G$, $Z$, $\delta$ the group $Z$ is cyclic of even order,
and $g:(G/\rad g)^3\to \{a\in Z;\;a^2=1\}$ is a nontrivial trilinear
alternating form.
\end{lemma}
\begin{proof}
Assume that $G$, $Z$, $\delta$ is a minimal setup and $Z$ is not cyclic. Then
$Z=Z_1\times\cdots\times Z_k$ for some cyclic groups $Z_i$ and $k\ge 2$. Let
$\pi_i:Z\to Z_i$ be the canonical projections. Since $g:G^3\to Z$ satisfies
$g\ne 1$, there is $j$ such that $\pi_j g:G^3\to Z_j$ satisfies $\pi_j g\ne 1$.

Let $Y = \prod_{i\ne j} Z_i$, $\overline{G}=G/Y$, $\overline{Z}=Z/Y\cong Z_j$,
and define $\overline{\delta}:\overline{G}^2\to\overline{Z}$ by
$\overline{\delta}(xY,yY) = \pi_j\delta(x,y)$. Then $\overline{G}$,
$\overline{Z}$, $\overline{\delta}$ is a new setup. Let $\overline{g}$ be
associated with $\overline{\delta}$ in a way analogous to \eqref{Eq:fgh}. Then
$\overline{g} = \pi_j g\ne 1$, a contradiction with the minimality of $|G|$.

Thus $Z$ is cyclic. Assume that it is of odd order. Then $U=\{a\in
Z;\;a^2=1\}=1$ and the setup is trivial by Lemma \ref{Lm:fgh}(v), a
contradiction.

Since $U=\{a\in Z;\;a^2=1\}\cong\mathbb Z_2$, we can view $G/\rad g$ as a
vector space over the two-element field and identify $g$ with a nontrivial
trilinear alternating form, by Lemma \ref{Lm:fgh}(v).
\end{proof}

\begin{lemma}\label{Lm:Reduction}
In a minimal setup, if $x$, $y$, $z\in G$ are such that $g(x,y,z)\ne 1$ then
$G=\langle x, y, z\rangle Z$.
\end{lemma}
\begin{proof}
Let $L=\langle x, y, z\rangle Z$. Since $Z\le Z(G)$, $Z\le L$ and $L\cap
Z(G)\le Z(L)$, we have $Z\le Z(L)$. Restrict $\delta$ to $L\times L$. Then $L$,
$Z$, $\delta$ is a new setup (i.e., the conditions \eqref{Eq:D0}--\eqref{Eq:D3}
hold with $L$ in place of $G$) and since $x$, $y$, $z\in L$, this setup is
nontrivial.
\end{proof}

\begin{lemma}\label{Lm:Factor}
In a minimal setup, let $L$ be a subgroup satisfying $Z\le L\unlhd G$. Then
$G/L$ is generated by at most $3$ elements. In particular, $G/\rad g$ is an
elementary abelian group of order $8$.
\end{lemma}
\begin{proof}
We know that $G/\rad g$ is an elementary abelian $2$-group by Lemma
\ref{Lm:fgh}(v) and we must have $\dim(G/\rad g)\ge 3$ to ensure $g\ne 1$. The
rest follows from Lemma \ref{Lm:Reduction}.
\end{proof}

\begin{lemma}\label{Lm:MRGR}
In a nontrivial setup, $|\mul\delta/\rad\delta|$ is even and $|G|$ is even.
Moreover, if $\mul\delta/\rad\delta$ is an elementary abelian $2$-group then
$G/\rad f$ is an elementary abelian $2$-group, too.
\end{lemma}
\begin{proof}
Suppose first that $\mul\delta/\rad\delta$ is of odd order $k$ and choose $x$,
$y$, $z\in G$ with $g(x,y,z)\ne 1$. Then $g(x,y,z) = g(x^k,y^k,z^k)$ by Lemma
\ref{Lm:fgh}(iii), (v). On the other hand, $[x^k,y^k](\rad\delta) =
([x,y]^{k^2})\rad\delta = \rad\delta$, since $G'\le\mul\delta$ and
$\cl{G/\rad\delta}\le 2$. Then $f(x^k,y^k,z^k) = \delta([x^k,y^k],z^k) = 1$,
thus $g(x^k,y^k,z^k)=1$, a contradiction.

Now suppose that $\mul\delta/\rad\delta$ is an elementary abelian $2$-group.
Then for every $x$, $y$, $z\in G$ we have $f(x^2,y,z) = \delta([x^2,y],z) =
\delta([x,y]^2,z) = 1$, and similarly $f(x,y^2,z) = 1$. This implies
$f(x,y,z^2) = g(x,y,z^2) = 1$ by Lemma \ref{Lm:fgh}(v).
\end{proof}

\begin{lemma}\label{Lm:RadfElem}
In a minimal setup, $\rad f = \rad g$ and $\im f\setminus \{1\}$ consists of
the unique involution of $Z$.
\end{lemma}
\begin{proof}
Suppose first that $G/\rad f$ is not an elementary abelian $2$-group. If
$|Z|=2$ then $f(x^2,y,z)=f(x,y,z)^2=1$ by Lemma \ref{Lm:fgh}, a contradiction.
Thus $|Z|\ge 4$ by Lemma \ref{Lm:ZCyclic}. By Lemma \ref{Lm:MRGR} we have
$|\mul\delta/\rad\delta|\ge 4$. By Lemma \ref{Lm:fgh}(v), $G/\rad g$ is
elementary abelian, and thus $|\rad g/\rad f|\ge 2$. Since $|G/\rad g|=8$ by
Lemma \ref{Lm:Factor}, we have $|G|>128$, and so the setup is not minimal.

Now suppose that $G/\rad f$ is elementary abelian. Then $\dim(G/\rad f)\le 3$
by Lemma \ref{Lm:Factor} and $\rad g = \rad f$ follows. Also, $1=f(x^2,y,z) =
f(x,y,z)^2$, so every nontrivial value of $f$ is an involution in $Z$, and this
involution is unique by Lemma \ref{Lm:ZCyclic}.
\end{proof}

We have shown that in a minimal setup we must have $|G/\rad g|=8$,
$|\mul\delta/\rad\delta|\ge 2$ even, and $|Z|\ge 2$ even. We proceed in two
directions, depending on whether $Z(G/\rad\delta)$ is a subgroup of $\rad
f/\rad\delta$ or not.

\subsection{$Z(G/\rad\delta)$ is a subgroup of $\rad f/\rad\delta$}

\begin{lemma}\label{Lm:H}
Let $H$ be a $2$-group of order $\ge 16$, and $H/Z(H)$ an elementary abelian
$2$-group of order $8$. Then $H'$ is an elementary abelian subgroup of $Z(H)$
and is of order $\ge 4$. If $|H'| = 4$, then there exist $u$, $v$, $w \in H$
that generate $H$ modulo $Z(H)$ and satisfy $ab = c$, $1 \notin \{a,b,c\}$,
where $a= [u,v]$, $b = [v,w]$ and $c = [w,u]$.
\end{lemma}
\begin{proof}
The mapping $[-,-]$ induces a nondegenerate alternating bilinear mapping of
$H/Z(H)$ into $H'\le Z(H)$. From $[x^2,y]=[x,y]^2=1$ we see that $H'$ has to be
elementary abelian. If $|H'|\le 2$ then $[-,-]:(H/Z(H))^2\to H'$ is an
alternating bilinear form with trivial radical and $\dim(H/Z(H))=3$, which is
impossible. Therefore $|H'|\ge 4$.

Assume that $H'$ is the Klein group, and let $H/Z(H)$ be generated by $x$, $y$,
$z$. Then $H'$ is generated by $[x,y]$, $[y,z]$ and $[z,x]$. Note that at most
one of $[x,y]$, $[y,z]$, $[z,x]$ is trivial, else the three elements do not
generate $H'$. The three elements cannot all be the same for the same reason.

If all $[x,y]$, $[y,z]$, $[z,x]$ are nontrivial and distinct, we automatically
have $[x,y][y,z]=[z,x]$ and are done. If all three are nontrivial and precisely
two coincide, say $[x,y]=[y,z]$, then $[xz,y]=[x,y][z,y]=1$ and $[y,z]$,
$[z,xz]=[z,x]$ are two distinct nontrivial elements of $H'$. We can therefore
assume that $[x,y]=1$ and $[y,z]$, $[z,x]$ are two distinct nontrivial elements
of $H'$. Then $u=zx$, $v=zy$, $w=z$ do the job, as $[u,v]=[zx,zy] =
[z,y][x,z][x,y] = [z,y][x,z] = [zy,z][z,zx] = [v,w][w,u]$ is not equal to $1$.
\end{proof}

\begin{lemma}[Baer]\label{Lm:Baer} Let $H$ be a group and $H/Z(H)$ an abelian group. For a
prime $p$ let $e(p)$ be the exponent of the $p$-primary component of $H/Z(H)$.
Then $\mathbb Z_{e(p)}\times \mathbb Z_{e(p)}\le H/Z(H)$.
\end{lemma}

\begin{proposition}\label{Pr:Case2}
In a minimal setup, let $\overline{G}=G/\rad\delta$ and assume that
$Z(\overline{G})\le\rad f/\rad\delta$. Then $Z(\overline{G})=\rad
f/\rad\delta=\overline{G}\,'$ is elementary abelian of order $8$ and $\rad f =
\mul\delta$.
\end{proposition}
\begin{proof}
Let $\overline{\rad f} = \rad f/\rad\delta$ and $\overline{\mul\delta} =
\mul\delta/\rad\delta$. From $Z(\overline{G})\le \overline{\rad f} <
\overline{G}$ we see that $\overline{G}$ is not abelian. Hence
$\cl{\overline{G}}=2$ by Lemma \ref{Lm:fgh}(i), and we have
$1<\overline{G}\,'\le Z(\overline{G})<\overline{G}$.

Recall that we have $|G|\le 128$, and thus $|\overline{G}|\le 64$. By Lemma
\ref{Lm:MRGR}, $|\overline{\mul\delta}|=2k$ for some $k\ge 1$, so
$|\overline{G}|$ is divisible by $16k$, which means that either $\overline{G}$
is a $2$-group or $|\overline{G}|=48=16\cdot 3$. Since $\overline{G}$ is
nilpotent, $\overline{G}\,'$ has to be a $2$-group in any case.

We claim that if $\overline{G}\,'$ is elementary abelian then
$\overline{G}/Z(\overline{G})$ is elementary abelian and $|\overline{G}\,'|\ge
4$. Indeed, we have $[x^2,y]=[x,y]^2=1$ so $\overline{G}/Z(\overline{G})$ is
elementary abelian, its order cannot exceed $8$ (by Lemma \ref{Lm:Factor}),
hence it is equal to $8$ (as $Z(\overline{G})\le\overline{\rad f}$), and so
$|\overline{G}\,'|\ge 4$ follows by Lemma \ref{Lm:H}.

If $|\overline{G}\,'|=2$, we have a contradiction with the claim. We can
therefore assume that $|\overline{G}\,'|\ge 4$, and thus also
$|Z(\overline{G})|\ge 4$.

Suppose for a while that $|Z(\overline{G})|=4$ and
$Z(\overline{G})=\overline{\rad f}$. Then $\overline{G}/Z(\overline{G})\cong
G/\rad f$ is elementary abelian of order $8$, and we are in the situation of
Lemma \ref{Lm:H} with $|\overline{G}\,'|=4$. Let $u$, $v$, $w$ be as in Lemma
\ref{Lm:H}, so $[v,w] = [u,v][u,w]r$ for some $r\in\rad\delta$. We have
$g(u,v,w) \ne 1$, and we can assume, say, $f(v,w,u) \ne 1$. Now, $f(u,v,u) =
h(u,v,u)$ and $f(u,w,u) = h(u,w,u)$. That means that $f(v,w,u) =
\delta([v,w],u) = \delta([u,v][u,w]r,u) = \delta([u,v],u)\delta([u,w],u)=
h(u,v,u)h(u,w,u) = [u,[v,u]][u,[w,u]] = [u,[v,w]] = h(u,v,w)$. However, that
yields  $g(u,v,w) = f(v,w,u)h(u,v,w) = 1$, a contradiction.

Now suppose that $|Z(\overline{G})|=4$ and $Z(\overline{G})<\overline{\rad f}$.
Then $|\overline{\rad f}|=8$ and $|\overline{G}/Z(\overline{G})|>8$, so
$\overline{G}/Z(\overline{G})$ cannot be elementary abelian by Lemma
\ref{Lm:Factor}. But $\overline{G}/\overline{\rad f}$ is elementary abelian,
$|\overline{\rad f}/Z(\overline{G})|=2$, a contradiction with Lemma
\ref{Lm:Baer}.

It remains to consider the situation $|\overline{G}\,'|\ge 4$,
$|Z(\overline{G})|=8$, $Z(\overline{G})=\overline{\rad f}$. Then
$\overline{G}/Z(\overline{G})\cong G/\rad f$ is elementary abelian of order
$8$. If $|\overline{G}\,'|=4$, we reach a contradiction by Lemma \ref{Lm:H} as
above. We therefore have $\overline{G}\,'=Z(\overline{G})$. Since
$G'\le\mul\delta$, we have $\overline{G}\,'\le \overline{\mul\delta}$, and
$\mul\delta=\rad f$ follows. As $\overline{G}/\overline{G}\,'$ is elementary
abelian, we must have $1=[x^2,y] = [x,y]^2$, which shows that $\overline{G}\,'$
itself is elementary abelian.
\end{proof}

\subsection{$Z(G/\rad\delta)$ is not a subgroup of $\rad f/\rad\delta$}

\begin{proposition}\label{Pr:CaseIII}
In a minimal setup, suppose that $Z(G/\rad\delta)$ is not contained in $\rad
f/\rad\delta$. Then $|\rad\delta/Z|\ge 4$. Moreover, if $|\rad\delta/Z|=4$ then
$\rad\delta/Z$ is the Klein group.
\end{proposition}
\begin{proof}
We will use Lemma \ref{Lm:fgh} freely in this proof. Let
$\overline{G}=G/\rad\delta$ and let $a\in G\setminus\rad f$ be such that
$a\rad\delta\in Z(\overline{G})$. Then $[a,r]\in\rad\delta$ for every $r\in G$,
and thus $f(r,a,s)^{-1} = f(a,r,s) = \delta([a,r],s) = 1$ for every $r$, $s\in
G$. Fix $x$, $y\in G$ such that $f(x,y,a)\ne 1$, and note that
$g(x,y,a)=f(x,y,a)\ne 1$. Also fix $b=[x,y]$.

Then $\delta(b,a) = f(x,y,a)\ne 1$. Moreover, $\delta(b,rs) =
\delta(b,r)\delta(b,s)$ for every $r$, $s\in G$, since $G'\le\mul\delta$. If
$s\in\rad f$, we get $\delta(b,rs) = \delta(b,r)\delta(b,s) =
\delta(b,r)f(x,y,s) = \delta(b,r)$. We thus consider $\delta(b,-)$ as a
homomorphism $G/\rad f\to \im f$. Since $|\im f|=2$ by Lemma \ref{Lm:RadfElem}
and $|G/\rad f|=8$, $\delta(b,-)$ has kernel of size $4$.

Fix $u$, $v\in G$ such that $\langle u,v,a\rangle\rad f = G$ and
$\delta(b,u)=\delta(b,v)=1$. We claim that $[u,v]\not\in\rad\delta$. Indeed,
should $[u,v]\in\rad\delta$, then $[r,s]\in\rad\delta$ for every $r$,
$s\in\{u,v,a\}$, and so $f(r,s,t)=\delta([r,s],t)=1$ for every $r$, $s$, $t\in
G$, a contradiction.

By Lemma \ref{Lm:Reduction}, $G = \langle x, y, a\rangle Z$ and so $G=\langle
x, y, a\rangle\rad \delta$, too. The group $\overline{G}\,'$ is then generated
by $\{[x,y]\rad\delta$, $[a,x]\rad\delta$, $[a,y]\rad\delta\} =
\{[x,y]\rad\delta\}$ and is therefore cyclic. Since $[u,v]\not\in\rad\delta$,
we have $[u,v]\rad\delta = [x,y]^m\rad\delta$ for some $m$. Then $f(u,v,u) =
\delta([u,v],u) = \delta([x,y]^m,u) = \delta(b,u)^m = 1$ and, similarly,
$f(u,v,v) = 1$. It follows that $f(r,s,t)=1$ for all $r$, $s$, $t\in\{u,v,a\}$,
except possibly for $f(u,v,a)=f(v,u,a)$. But then we must have
$f(u,v,a)=f(v,u,a)\ne 1$, else $f$ is trivial.

We claim that $[a,u^i]= [a,u]^i$ for every $i$. The claim is certainly true for
$i=1$. The group identity $[r,st] = [r,t][r,s][[r,s],t]$ yields $[a,u^iu] =
[a,u][a,u^i][[a,u^i],u]$, so it suffices to show that $[[a,u^i],u]=1$. Now,
$[[a,u^i],u] = [u,[a,u^i]]^{-1} = h(u,a,u^i)^{-1} = h(u,a,u)^{-i}$, and
$h(u,a,u) = h(u,a,u)f(u,a,u) = f(u,u,a) = 1$. Similarly, $[a,v^i]=[a,v]^i$ for
every $i$.

We also claim that
\begin{equation}\label{Eq:Claim}
    [a,u]^iZ \ne [a,v]Z\text{ and }[a,v]^iZ\ne [a,u]Z\text{ for every $i$}.
\end{equation}
Indeed, if $[a,u]^i=[a,v]z$ for some $z\in Z$, we have $1=h(u,a,u) = h(u,a,u)^i
= h(u,a,u^i) = [u,[a,u^i]] = [u,[a,v]z] = [u,[a,v]] = f(u,v,a)f(u,a,v) =
f(u,v,a)\ne 1$, a contradiction. The other case is similar.

By \eqref{Eq:Claim}, $[a,u]\not\in Z$, $[a,v]\not\in Z$, and $[a,v]Z\ne
[a,u]Z$. If $[a,u]^2\not\in Z$, \eqref{Eq:Claim} yields $|\rad\delta/Z|\ge 4$,
as desired. If $[a,u]^2\in Z$ then $|\rad\delta/Z|$ is even, and hence also
$|\rad\delta/Z|\ge 4$.

Finally, let us assume that $\rad\delta/Z$ is a cyclic group of order $4$. Then
one of $[a,u]Z$, $[a,v]Z$ generates $\rad\delta/Z$, a contradiction with
\eqref{Eq:Claim}.
\end{proof}

\begin{lemma}\label{Lm:GModZCaseIII}
In a minimal setup, suppose that $Z(G/\rad\delta)$ is not a subgroup of $\rad
f/\rad\delta$. Let $K=G/Z$. Then $|K'|\ge 8$. If $|\rad f/Z|=8$ then
$K'=Z(K)=\rad f/Z$.
\end{lemma}
\begin{proof}
With the notation of the proof of Proposition \ref{Pr:CaseIII}, $\rad\delta/Z$
is the Klein group generated by $[a,u]Z$, $[a,v]Z$, and the commutator $[u,v]Z$
does not belong to $\rad\delta$, which implies $|K'|\ge 8$. For the rest of the
proof assume that $|\rad f/Z|=8$.

As $G'\le\mul\delta\le\rad f$ and $|K'|\ge 8$, we must have $K'\ge \rad f/Z$.
On the other hand, $G/\rad f\cong (G/Z)/(\rad f/Z)$ is an abelian group by
Lemmas \ref{Lm:Factor} and \ref{Lm:RadfElem}, so $K'=\rad f/Z$.

Since $\cl{K}=2$ by Lemma \ref{Lm:fgh}, we have $K'\le Z(K)$. For the other
inclusion, let $r\in Z(K)$. Then $[r,s]\in Z$ for every $s\in G$, so $f(s,r,t)
= f(r,s,t) = \delta([r,s],t) = 1$ and $f(s,t,r) = [s,[r,t]]f(s,r,t) =
[s,[r,t]]=1$ for every $r$, $t\in G$. Hence $r\in \rad f/Z = K'$.
\end{proof}

\section{The minimal setups}\label{Sc:Min}

We will need the following two lemmas concerning small $2$-groups:

\begin{lemma}\label{Lm:K}
Let $K$ be a group such that
\begin{equation}\label{Eq:GroupK}
    |K|=64,\,|K'|=8,\text{ and }K'=Z(K).
\end{equation}
Then both $K'$ and $K/K'$ are elementary abelian of order $8$. In addition,
there are $e_1$, $e_2$, $e_3\in K$ such that $\{e_1K'$, $e_2K'$, $e_3K'\}$ is a
basis of $K/K'$, and $\{[e_1,e_2]$, $[e_1,e_3]$, $[e_2,e_3]\}$ is a basis of
$K'$.
\end{lemma}
\begin{proof}
The group $K/K' = K/Z(K)$ is elementary abelian, else $|K/Z(K)|>8$ by Lemma
\ref{Lm:Baer}, a contradiction. For all $x$, $y \in K$ we then get $[x,y]^2 =
[x,y^2]=1$, and so $K'$ is elementary abelian too. The rest is clear.
\end{proof}

\begin{lemma}\label{Lm:G128}
Let $G$ be a group such that $|G|=128$ and $\cl{G}=2$. Then $|G'|\le 8$. If
also $|G'|=8$ then $|Z(G)|\le 16$.
\end{lemma}
\begin{proof}
The commutator can be seen as a bilinear mapping $G/Z(G)\to G'$. The group
$G/Z(G)$ cannot be cyclic. Suppose for a while that $|G'|>8$. Then $|G/Z(G)|\le
8$. Assume that $G/Z(G)$ is elementary abelian of order $8$. Then there are
$e_1$, $e_2$, $e_3\in G$ such that every commutator is of the form
$[e_1^{a_1}e_2^{a_2}e_3^{a_3},e_1^{b_1}e_2^{b_2}e_3^{b_3}]$, which is a product
of $[e_1,e_2]$, $[e_1,e_3]$ and $[e_2,e_3]$ thanks to $\cl{G}=2$. Since
$[e_i,e_j]^2 = [e_i,e_j^2] = 1$ by $e_j^2\in Z(G)$, we see that $G'\le Z(G)$ is
an elementary abelian $2$-group generated by three elements, so $|G'|\le 8$, a
contradiction. We can argue similarly when $G/Z(G)\cong \mathbb Z_4\times
\mathbb Z_2$ or when $G/Z(G)$ is elementary abelian of order $4$.

Hence $|G'|\le 8$. Suppose that $|G'|=8$. To show that $|Z(G)|\le 16$, it
suffices to prove that $G/Z(G)\cong\mathbb Z_2\times\mathbb Z_2$ is impossible.
This is once again easy.
\end{proof}

We can now summarize our results on minimal setups:

\begin{theorem}\label{Th:Scenarios}
Let $G$, $Z$, $\delta$ be a minimal setup. Then $|G|$ is even and $|G|\ge 128$.
If $|G|=128$ then $|Z|=2$, $\rad g = \rad f = \mul\delta = G'Z$,
$\mul\delta/\rad\delta$ is an elementary abelian $2$-group, $G/\rad g$ is an
elementary abelian group of order $8$, $K=G/Z$ satisfies \eqref{Eq:GroupK}, and
one of the following scenarios holds, with $\overline{G}=G/\rad\delta$:
\begin{enumerate}
\item[(i)] $\cl{G}{=}2$, $G'<\mul\delta = Z(G)$, $Z(\overline{G}) = \rad
    f/\rad\delta = \overline{G}\,'$, $|\mul\delta/\rad\delta|=8$,
    $\rad\delta=Z$, $f=g$, or
\item[(ii)] $\cl{G}{=}3$, $Z=[G,G']$, $G'=\mul\delta$, $Z(\overline{G}) =
    \rad f/\rad\delta = \overline{G}\,'$, $|\mul\delta/\rad\delta|=8$,
    $\rad\delta=Z$, or
\item[(iii)] $\cl{G}{=}3$, $Z=[G,G']\le G'=\mul\delta$, $Z(\overline{G})$
    is not a subgroup of $\rad f/\rad\delta$, $|\mul\delta/\rad\delta|=2$,
    $\rad\delta/Z$ is the Klein group.
\end{enumerate}
\end{theorem}
\begin{proof}
We know that $|G|=128$ can occur thanks to the examples constructed already in
\cite{DV1}. Assume that $|G|\le 128$. Then $G/\rad g$ is an elementary abelian
group of order $8$ by Lemma \ref{Lm:Factor}, $|\mul\delta/\rad\delta|\ge 2$ by
Lemma \ref{Lm:MRGR}, and $\rad g = \rad f$ by Lemma \ref{Lm:RadfElem}. Let
$K=G/Z$.

If $Z(\overline{G})\le\rad f/\rad\delta$ then $Z(\overline{G})=\rad
f/\rad\delta=\overline{G}\,'$ is elementary abelian of order $8$ and $\rad
f=\mul\delta$ by Proposition \ref{Pr:Case2}. This implies $\rad\delta=Z$,
$|Z|=2$, and $|G|=128$. By Proposition \ref{Pr:Case2} again, $K$ satisfies
\eqref{Eq:GroupK}.

If $Z(\overline{G})$ is not a subgroup of $\rad f/\rad\delta$ then
$|\rad\delta/Z|\ge 4$ by Proposition \ref{Pr:CaseIII} and hence $\rad f =
\mul\delta$, $|\mul\delta/\rad\delta|=2$, $|\rad\delta/Z|=4$, $|Z|=2$ and
$|G|=128$. By Proposition \ref{Pr:CaseIII}, $\rad\delta/Z$ is then the Klein
group. Since $|\rad f/Z|=8$, Lemma \ref{Lm:GModZCaseIII} yields
\eqref{Eq:GroupK}.

In either case, let $M=\rad g = \rad f = \mul\delta$. As \eqref{Eq:GroupK}
holds, there are $e_1$, $e_2$, $e_3\in G$ such that $\{e_1M$, $e_2M$, $e_3M\}$
is a basis of $G/M$, and $\{[e_1,e_2]Z$, $[e_1,e_3]Z$, $[e_2,e_3]Z\}$ is a
basis of $M/Z$. Using $|Z|=2$ and Lemma \ref{Lm:fgh}, we have
\begin{align*}
    f(e_2,e_3,e_1) &= [e_2,[e_1,e_3]]f(e_2,e_1,e_3) = [e_2,[e_1,e_3]]f(e_1,e_2,e_3),\\
    f(e_3,e_1,e_2) & = f(e_1,e_3,e_2) = [e_1,[e_2,e_3]]f(e_1,e_2,e_3).
\end{align*}
Thus
\begin{equation}\label{Eq:gForcesf}
    1\ne g(e_1,e_2,e_3) = [e_2,[e_1,e_3]][e_1,[e_2,e_3]]f(e_1,e_2,e_3).
\end{equation}
If $Z(\overline{G})$ is not a subgroup of $\rad f/\rad\delta$ then
$|\rad\delta/Z|=4$ and we can assume without loss of generality that
$[e_1,e_2]\rad\delta = [e_1,e_3]\rad\delta$. Then
\begin{multline*}
    f(e_1,e_2,e_3) = \delta([e_1,e_2],e_3) = \delta([e_1,e_3],e_3) =
    f(e_1,e_3,e_3)\\
    = f(e_3,e_1,e_3) = [e_3,[e_3,e_1]]f(e_3,e_3,e_1) =
    [e_3,[e_3,e_1]]\delta([e_3,e_3],e_1) = [e_3,[e_3,e_1]]
\end{multline*}
and therefore
\begin{equation}\label{Eq:NecessaryCaseIII}
    1\ne g(e_1,e_2,e_3) = [e_2,[e_1,e_3]][e_1,[e_2,e_3]][e_3,[e_3,e_1]].
\end{equation}

We have $G'\le\mul\delta$, $Z\le\mul\delta$, so $G'Z\le\mul\delta$. Since
$|Z|=2$ and $[G,G']\le Z$, the following three conditions are equivalent:
$[G,G']<Z$, $[G,G']=1$, $\cl{G}=2$.

Suppose that $\cl{G}=2$. Then $Z(\overline{G})\le \rad f/\rad\delta$, since the
other alternative implies \eqref{Eq:NecessaryCaseIII}, a contradiction with
$[G,G']=1$. Furthermore, $|G'|\le 8$ by Lemma \ref{Lm:G128}. Since
$|K'|=|G'Z/Z|=8$, we have $|G'Z|=16$, and so $|G'|=8$, $G'<\mul\delta$,
$Z(G)\ge G'Z=\mul\delta$. By the second part of Lemma \ref{Lm:G128},
$Z(G)=\mul\delta$. Finally, $f(x,y,z) = f(y,z,x) = f(x,z,y)$ by Lemma
\ref{Lm:fgh} and $\cl{G}\le 2$, so $g(x,y,z) = f(x,y,z)^3=f(x,y,z)$.

Now suppose that $\cl{G}=3$. Then $Z=[G,G']\le G'$, so $K' = (G/Z)' = G'Z/Z =
G'/Z$, which implies $|G'|=16$ and thus $G'=G'Z = \mul\delta$. The rest of
(ii), (iii) has already been established.
\end{proof}

It is now easy to characterize all groups $G/Z$ from minimal setups. Note that
these groups are of interest for the associated loops $Q=G[\mu]$, too, since
$G/Z\cong Q/Z$.

\begin{proposition}\label{Pr:GModZ}
A group $K$ appears as $G/Z$ in a minimal setup if and only if $K$ satisfies
\eqref{Eq:GroupK}. All such groups appear already in scenario (i) of Theorem
\ref{Th:Scenarios}.
\end{proposition}
\begin{proof}
By Theorem \ref{Th:Scenarios}, $K=G/Z$ from a minimal setup satisfies
\eqref{Eq:GroupK}. Conversely, assume that $K$ satisfies \eqref{Eq:GroupK}.
Then $K=G/Z$ for some $G$ in scenario (i) by the results of \cite[\S 5]{DV1}.
\end{proof}

\subsection{Constructing minimal setups for scenarios (ii) and (iii)}

We show how to construct all minimal setups. Note that all minimal setups of
scenario (i) of Theorem \ref{Th:Scenarios} were constructed already in \cite[\S
5]{DV1}, so it suffices to work with scenarios (ii) and (iii).

First we obtain a few auxiliary facts about minimal setups.

\begin{lemma}\label{Lm:gfDetermined}
Let $G$, $Z$, $\delta$ be a minimal setup. Then for every basis
$\{e_1,e_2,e_3\}\subseteq G$ of $G/\rad g$ we have \eqref{Eq:gForcesf}. In
particular, $g$ and $f$ are determined already by $G$ and $\rad g$.
\end{lemma}
\begin{proof}
By Theorem \ref{Th:Scenarios}, $G/\rad g$ is an elementary abelian group of
order $8$. Let $\{e_1\rad g$, $e_2\rad g$, $e_3\rad g\}$ be a basis of $G/\rad
g$. Since $g$ is nontrivial, it is the determinant, so $g(e_i,e_j,e_k)\ne 1$ if
and only if $i$, $j$, $k$ are distinct. This determines $g$ as a mapping
$G\times G\times G\to Z$. We have derived \eqref{Eq:gForcesf} in the proof of
Theorem \ref{Th:Scenarios}. It remains to show that $f$ is determined by $g$
and $G$. Indeed, $1=f(e_i,e_i,e_i)$, $1=f(e_i,e_i,e_j)$ determines
$f(e_i,e_j,e_i)$ and $f(e_j,e_i,e_i)$ by Lemma \ref{Lm:fgh}, $f(e_1,e_2,e_3)$
is determined by \eqref{Eq:gForcesf}, and this value determines
$f(e_i,e_j,e_k)$ whenever $i$, $j$, $k$ are distinct.
\end{proof}

Let $G$, $Z$, $\delta$ be a minimal setup from scenario (ii) or (iii), and let
$R=\rad\delta$, $M=\mul\delta = \rad f = \rad g$. Note that $M=N=G'R=G'$ here,
by Theorem \ref{Th:Scenarios}.

By Lemma \ref{Lm:K}, there are $e_1$, $e_2$, $e_3\in G$ such that $\{e_1M$,
$e_2M$, $e_3M\}$ is a basis for $G/M$, and $\{[e_1,e_2]Z$, $[e_1,e_3]Z$,
$[e_2,e_3]Z\}$ is a basis for $M/Z$.

In scenario (ii), $Z=R$, so we have a basis for $M/R$. In scenario (iii),
$|R/Z|=4$, and we can therefore assume without loss of generality that
$[e_1,e_2]R = [e_1,e_3]R$.

\medskip

We finally turn to the construction of all minimal setups. Let us therefore
forget about $\delta$, $f$ and $g$, but let us keep the groups $G$, $M$, $R$,
$Z$ and the elements $e_1$, $e_2$, $e_3$. Our goal is to construct a nontrivial
setup $G$, $Z$, $\delta$ with $M=\mul\delta$ and $R=\rad\delta$. Let $Z=\{1$,
$-1\}$.

First of all, the mapping $g:(G/M)^3\to Z$ must be a trilinear alternating
form, and hence we must and can set
\begin{displaymath}
    g(e_i,e_j,e_k) = \left\{\begin{array}{rl}
        -1,&\text{ if $i$, $j$, $k$ are distinct,}\\
        1,&\text{ else,}
    \end{array}\right.
\end{displaymath}
and then extend $g$ linearly.

Next we need a multiplicative mapping $f:(G/M)^3\to Z$ such that $g(x,y,z) =
f(x,y,z)f(y,z,x)f(z,x,y)$ and such that $f$ behaves as in Lemma \ref{Lm:fgh}.
Anticipating the equality $\delta([x,y],z) = f(x,y,z)$, we must set
\begin{displaymath}
    f(e_i,e_i,e_j)=1\text{ for $1\le i$, $j\le 3$}.
\end{displaymath}
Then Lemma \ref{Lm:fgh} forces
\begin{displaymath}
    f(e_i,e_j,e_i) = f(e_j,e_i,e_i) = [e_i,[e_i,e_j]]\text{ for $1\le i$, $j\le
    3$}.
\end{displaymath}
By Lemma \ref{Lm:gfDetermined}, we must set
\begin{displaymath}
    f(e_1,e_2,e_3) = f(e_2,e_1,e_3) = - [e_1,[e_2,e_3]][e_2,[e_1,e_3]],
\end{displaymath}
and then Lemma \ref{Lm:fgh} forces
\begin{align*}
    f(e_1,e_3,e_2)&=f(e_3,e_1,e_2) = -[e_2,[e_1,e_3]],\\
    f(e_2,e_3,e_1)&=f(e_3,e_2,e_1) = -[e_1,[e_2,e_3]].
\end{align*}
A straightforward calculation yields $g(e_i,e_j,e_k) =
f(e_i,e_j,e_k)f(e_j,e_k,e_i)f(e_k,e_i,e_i)$ for every $1\le i$, $j$, $k\le 3$.

We can now extend $f$ linearly into a mapping $(G/M)^3\to Z$, and force $\rad
f = M$.

Finally, we need to construct $\delta:(G/R)^3\to Z$ so that
$\delta([x,y],z)=f(x,y,z)$ and \eqref{Eq:D0}--\eqref{Eq:D3} hold. We will
encounter a difficulty in scenario (iii), which is why we only managed to
answer some questions concerning scenario (iii) using a computer.

Set
\begin{displaymath}
    \delta([e_i,e_j],e_k) = f(e_i,e_j,e_k)\text{ for $1\le i$, $j$, $k\le 3$}.
\end{displaymath}
Since $G'=M$, we can now attempt to extend $\delta$ into a mapping
$M/R\times\{e_1,e_2,e_3\}\to Z$. This unique extension is well defined in
scenario (ii), as the values $[e_1,e_2]$, $[e_1,e_3]$, $[e_2,e_3]$ are linearly
independent modulo $R$. But in scenario (iii) the extension might not exist,
and this can be verified with a computer in each particular case.

Assuming that $\delta:M/R\times\{e_1,e_2,e_3\}\to Z$ is well-defined, we extend
it routinely into a mapping $M/R\times G/R\to Z$, using $M=\rad f$.

The last step is to extend $\delta$ into a mapping $G/R\times G/R\to Z$, and
this involves some free parameters. Namely, let $T=\{t_1=1$, $\dots$, $t_n\}$
be a transversal to $M/R$ in $G/R$, and for $1\le i$, $j\le 3$ choose
$\delta(t_i,t_j)$ as follows:
\begin{align}
    &\delta(t_1,t_j) = 1 \text{ for every $1\le j\le n$},\notag\\
    &\delta(t_i,t_j) \text{ arbitrary when $1<i<j\le n$},\notag\\
    &\delta(t_j,t_i)=\delta(t_i,t_j)^{-1} \text{ when $1<i<j\le n$},\label{Eq:DeltaParams}\\
    &\delta(t_i,t_i)=1 \text{ for every $1\le i\le n$.}\notag
\end{align}
Every element $h\in G/R$ can be written uniquely as $h=mt$ for some $m\in M/R$,
$t\in T$. We define $\delta:G/R\times G/R\to Z$ by
\begin{displaymath}
    \delta(mt,m't') = \delta(m,t')\delta(m',t)^{-1}\delta(t,t'),
\end{displaymath}
where $\delta(m,t')$, $\delta(m',t)$ have already been defined above. We leave
it to the reader to check that this correctly defines $\delta:G\times G\to Z$
satisfying \eqref{Eq:D0}--\eqref{Eq:D3} and $\delta([x,y],z) = f(x,y,z)$.

We have arrived at a minimal setup $G$, $Z$, $\delta$.

\subsection{The groups $G$ in minimal setups}\label{Ss:ComputationalResults}

For the sake of completeness, we now describe the groups $G$ and $G/Z$ that
appear in individual scenarios of Theorem \ref{Th:Scenarios}:
\begin{enumerate}

\item[$\bullet$] The groups $G/Z$ of scenario (i) (respective (ii)) are
    precisely the groups $K$ satisfying \eqref{Eq:GroupK}. There are $10$
    such groups, identified as $(64,73)$--$(64,82)$ in GAP \cite{GAP}.

\item[$\bullet$] The groups $G/Z$ of scenario (iii) are precisely the
    groups $(64,73)$--$(63,76)$ and $(64,80)$ of GAP, by computer search.

\item[$\bullet$] The groups $G$ of scenario (i) are precisely the groups
    $G$ such that: $\cl{G}=2$, there is $Z\le Z(G)$ such that $|Z|=2$ and
    $K=G/Z$ satisfies \eqref{Eq:GroupK}. There are $19$ such groups,
    identified in GAP as $(128,m)$ for $m\in\{$170--178, 1116--1119,
    1121--1123, 1125, 1126, 1132$\}$.

\item[$\bullet$] The groups $G$ of scenario (ii) are precisely the groups
    $G$ such that: $\cl{G}=3$, $Z=[G,G']\le Z(G)$, $|Z|=2$, and $K=G/Z$
    satisfies \eqref{Eq:GroupK}. There are $106$ such groups, identified as
    $(128,731)$--$(128,836)$ in GAP.

\item[$\bullet$] The groups $G$ of scenario (iii) are precisely the $10$
    groups $(128,m)$ for $m\in\{742$, $749$, $753$, $754$, $761$, $762$,
    $776$, $794$, $823$, $830\}$ of GAP, by computer search.
\end{enumerate}

\section{Examples of loops of Cs\"org\H{o} type}\label{Sc:Examples}

Very many examples of loops of Cs\"org\H{o} type can be constructed from
Theorem \ref{Th:Scenarios} due to the free parameters (in
\eqref{Eq:DeltaParams} and \eqref{Eq:CPS}) used in the top to bottom
construction of $\mu$ from $g$. All examples listed below were obtained with
the LOOPS \cite{LOOPS} package for GAP \cite{GAP} by using trivial values for
all free parameters.

To get a loop of Cs\"org\H{o} type as in scenario (ii) of Theorem
\ref{Th:Scenarios}, let $G$ be the group of order $128$ generated by $g_1$,
$\dots$, $g_7$ with $Z(G)=\langle g_5, g_6, g_7\rangle$ subject to the
relations $g_1^2=1$, $[g_2,g_1]=g_4$, $[g_3,g_1]=g_5$, $[g_4,g_1]=g_7$,
$[g_3,g_2]=g_6$, $[g_4,g_2]=g_7$, $g_3^2=1$, $[g_4,g_3]=1$, $g_4^2=g_7$,
$g_5^2=1$, $g_6^2=1$, $g_7^2=1$. (This is the group identified as $(128,731)$
in GAP.)

Then $G'=\langle g_4,Z(G)\rangle$ and $\cl{G}=3$. Set $Z=R=\langle g_7\rangle$,
$M=G'$, $e_1=g_1$, $e_2=g_2$, $e_3=g_3$. Then the resulting loop $Q = G[\mu]$
satisfies $|Q|=128$, $\cl{Q}=3$, $\inn Q\cong \mathbb Z_4\times\mathbb
Z_4\times\mathbb Z_2\times\mathbb Z_2$, $|\mlt{Q}|=8192$.

To get a loop of Cs\"org\H{o} type as in scenario (iii) of Theorem
\ref{Th:Scenarios}, let $G$ be the group of order $128$ generated by $g_1$,
$\dots$, $g_7$ with $Z(G) = \langle g_4g_5g_6,g_7\rangle$ subject to relations
$g_1^2=1$, $[g_2,g_1]=g_4$, $[g_3,g_1]=g_5$, $[g_4,g_1]=g_7$, $[g_5,g_1]=1$,
$[g_6,g_1]=g_7$, $g_2^2=1$, $[g_3,g_2]=g_6$, $[g_4,g_2]=g_7$, $[g_5,g_2]=g_7$,
$[g_6,g_2]=1$, $g_3^2=1$, $[g_4,g_3]=1$, $[g_5,g_3]=1$, $[g_6,g_3]=1$,
$g_4^2=g_7$, $[g_5,g_4]=1$, $[g_6,g_4]=1$, $g_5^2=1$, $[g_6,g_5]=1$, $g_7^2=1$.
(This is the group identified as $(128,742)$ in GAP.)

Then $G'=\langle g_4,g_5,g_6,g_7\rangle$ and $\cl{G}=3$. Set $Z = \langle
g_7\rangle$, $R = \langle g_5,g_6,g_7\rangle$, $M=G'$, and $e_1=g_1$,
$e_2=g_2$, $e_3=g_2g_3$. Then $R/Z$ is the Klein group, and the resulting loop
$Q = G[\mu]$ satisfies $|Q|=128$, $\cl{Q}=3$, $\inn Q\cong \mathbb
Z_4\times\mathbb Z_4\times\mathbb Z_2\times\mathbb Z_2$, $|\mlt{Q}|=8192$.

\subsection{A class of examples with $\inn{Q}$ not elementary abelian}

\begin{lemma}\label{Lm:T2}
Let $Q=G[\mu]$ be constructed from a minimal setup. For $x\in Q$, let $T_x =
R_x^{-1}L_x$ be the conjugation by $x$ in $Q$. Then $(T_x)^2 y =
y^{x^2}f(y,x,x)$ for every $x$, $y\in Q$.
\end{lemma}
\begin{proof}
By Theorem \ref{Th:Scenarios} we have $|Z|=2$. By \cite[Lemma 4.4]{DV1},
\begin{displaymath}
    T_x(y) = y^x\mu(y,x)\mu(x,y^x)^{-1} = y^x\mu(y,x)\mu(x,y^x).
\end{displaymath}
Then
\begin{displaymath}
    (T_x)^2y = T_x(y^x\mu(y,x)\mu(x,y^x)) =
    y^{x^2}\mu(y,x)\mu(x,y^x)\mu(y^x,x)\mu(x,y^{x^2}),
\end{displaymath}
where we have used $\im{\mu}\le Z\le Z(G)$ and $Z\le\rad{\mu}$. Now,
$\mu(x,y^x)\mu(y^x,x) = \delta(y^x,x) = \delta(y[y,x],x) =
\delta(y,x)\delta([y,x],x) = \delta(y,x)f(y,x,x)$ by $G'\le\mul\delta$, and
thus $(T_x)^2y = y^{x^2}\mu(y,x)\delta(y,x)f(y,x,x)\mu(x,y^{x^2})$.
Furthermore, $\mu(y,x)\delta(y,x)\mu(x,y^{x^2}) = \mu(x,y)\mu(x,y[y,x^2]) =
\mu(x,y)\mu(x,y)\mu(x,[y,x^2]) = \mu(x,[y,x^2])$, and $[y,x^2] = [y,x]^2z$ for
some $z\in Z$, so $\mu(x,[y,x^2]) = \mu(x,[y,x]^2) = \mu(x,[y,x])^2 = 1$.
\end{proof}

By Lemma \ref{Lm:T2}, in order to obtain $Q=G[\mu]$ from a minimal setup so
that $\inn{Q}$ is not elementary abelian, it suffices to choose $K=G/Z$ as one
of the groups satisfying \eqref{Eq:GroupK}, take $G$ as a central extension of
$K$ by the cyclic group of order $2$ so that $[y,x^2]\ne 1$ (hence $y^{x^2}\ne
y$) and $[x,[x,y]]=1$ (hence $f(y,x,x) = f(x,x,y)[x,[x,y]] = [x,[x,y]]=1$) for
some $x$, $y\in G$.

\end{document}